\documentclass{article}   	

\usepackage{geometry, amsmath,mathrsfs}                		
\usepackage{graphicx}				
\usepackage{indentfirst}								
\usepackage{amssymb}
\usepackage{cite}
\usepackage[utf8]{inputenc}
\usepackage{amsfonts}
\usepackage{tikz}

\usepackage{pgfplots}
\usepackage{cite}
\usepackage{bm}
\usepackage{bbold}
\usepackage[affil-it]{authblk}
\usepackage[colorlinks,
            linkcolor=blue,
            anchorcolor=green,
            citecolor=red
            ]{hyperref}

\allowdisplaybreaks

\numberwithin{equation}{section}

\newtheorem{theorem}{Theorem}[section]
\newtheorem{lemma}[theorem]{Lemma}

\newenvironment{proof}[1][Proof]{\begin{trivlist}
\item[\hskip \labelsep {\bfseries #1}]}{\end{trivlist}}
\newenvironment{remark}[1][Remarks]{\begin{trivlist}
\item[\hskip \labelsep {\bfseries #1}]}{\end{trivlist}}

\newcommand*{\QEDA}{\hfill\ensuremath{\blacksquare}}%

\begin{document}

\title{Occupation times of discrete-time fractional Brownian motion}

\author{Manfred Denker\footnote{mhd13@psu.edu} , Xiaofei Zheng  \footnote{xxz145@psu.edu}}
\affil{Department of Mathematics, The Pennsylvania State University, \\State College, PA, 16802, USA}
\date{}
\maketitle
\begin{abstract}
We prove a conditional local limit theorem for  discrete-time fractional Brownian motions (dfBm) with Hurst parameter $\frac 34<H<1$. Using results from infinite ergodic theory it is then shown that the properly scaled occupation time of dfBm  converges to  a Mittag-Leffler distribution.
\end{abstract}
\section{Introduction and Main Results}

In 1957, Darling and Kac \cite{KacDarling} established limit theorems for the occupation times of Markov processes with stationary transition probabilities, proving that under a ``Darling-Kac condition'',  the limit distributions are necessarily Mittag-Leffler distributions with appropriate indices. Earlier  and weaker results were obtained by Dobrushin \cite{Do} and Chung and Kac \cite{CK}. The theory is applicable to Markov chains and in particular, to random walks, the sum of independent, identically distributed random variables $\{X_n\}$ with common distribution function $F$. $X_n$ may take a lattice  or  a non-lattice distribution.  When $F$ has mean $0$ and belongs to the domain of attraction of some stable law with index $d$,  it is known that $S_n$, the  partial sum of $\{X_n\}$ obeys a  local limit theorem \cite{Ibra}, which plays an important  role in proving the limiting distribution of the occupation time. This is because the local limit theorem  implies the ``Darling-Kac condition'' of  i.i.d. random variables $\{X_n\}$: $\sum_{k=0}^n P(S_k\in A)\sim |A|n^{1-\frac{1}{d}}L(n)$, where $L(n)$ is some slowly varying function\footnote{In the whole article, $A_n\sim B_n$ means $\frac{A_n}{B_n}\to 1$ as $n\to \infty$.}.  By Darling and Kac \cite{KacDarling}, the ``Darling-Kac condition'' implies that the normalized occupation time of $S_n$ converges to a Mittag-Leffler distribution. 

The Darling-Kac theorem can also be extended to the sum of weakly dependent random variables, one of such dependence is the Renyi-mixing sequence \cite{CS}.  Only little seems to be known on this topic, a recent study of the two authors \cite{Zheng,DZ} sheds some light onto this question by connecting  occupation times of ergodic sums in  Gibbs-Markov dynamical systems  to group extensions over these systems (see \cite{AD}).  The present paper is about the limiting distribution of occupation times of  the  discrete-time fractional Brownian motion $(B^H(n))_{n\in \mathbb Z}$ which is not weakly dependent but has long range dependence. It will  be shown that $B^H(n)$ satisfies a conditional local limit theorem.  Next, the occupation times  will be represented as partial sums of  iteratives of a transformation $\tilde{T}$ on the infinite measure space $\Omega\times \mathbb{R}$, where $(\Omega, \mathcal B_0, P)$ is the probability space carrying the Gaussian process. The conditional local limit theorem ensures that  $\tilde{T}$ is  a pointwise dual ergodic transformation with respect to the canonical product measure  on the product space $\Omega\times \mathbb{R}$. Finally, the limiting distribution of the occupation time can be shown to be   Mittag-Leffler distribution by  \cite{JonAsy}.

To be more precise, let $\{B^H(t)\}$ be a fractional Brownian motion with Hurst index $\frac{1}{2}<H<1$, and define $\{X_n\}$ to be the increment of the fBm: $X_n=B^H(n)-B^H(n-1)$ for $n\in \mathbb{Z^+}$. $\{X_n\}$ is  also called discrete-time fractional Gaussian noise (DFGN).  Then for any $n$, $(X_1,X_2,\cdots,X_n)$ has a multivariate normal distribution with mean $0$ and the covariance function $b(i-j)=E[X_iX_j]=E[X_jX_i]$ satisfying
\begin{equation*}\label{con1}
b(t)=\dfrac{1}{2}[(t+1)^{2H}-2t^{2H}+(t-1)^{2H}],
\end{equation*}
and $b(0)=E(X_i)^2=1$.
This is because $E[B^H(s)B^H(t)]=\frac{1}{2}[|s|^{2H}+|t|^{2H}-|t-s|^{2H}]$. So $\{X_n\}$ are Gaussian random variables  with a long-range dependence structure.

It is  proved in \cite{TM} that for  $\overline{d}_N^2\sim N^{2H}$,  $Z_N(t):=\dfrac{1}{\overline{d}_N}\sum_{i=1}^{[Nt]}X_i=\dfrac{1}{\overline{d}_N}B^H([Nt])$ converges weakly  to $B^H(t)$ in the Skorohod space $\mathcal{D}([0,1])$ as $N\to \infty$. Denote by $S_n$ the partial sum of $\{X_n\}$:  $S_n:=\sum_{i=1}^nX_i,$ which  actually equals $B^H(n)$.

Let $V$ be a non-negative function over $\mathbb{R}$, the state space of the DFGN $\{X_n\}$. In this article, we  study the limiting distribution of the random variable $\sum_{i=1}^n V(S_i)$ as $n\to \infty$.
If $V(x)$ is the characteristic function of some Borel set $B\subset \mathbb{R}$, then $\sum_{i=1}^n V(S_i)$ becomes the occupation time of $S_n$ of the set $B$: i.e. $\sum_{i=1}^n V(S_i)=\#\{i\leq n: S_i\in B\}$.
What plays an important role in finding the limiting distribution of the occupation time is the ``Darling-Kac'' condition: $\{S_n\}$ has a conditional local limit theorem when $\frac{3}{4}<H<1$. We state it below as our first main result.

\newpage

\begin{theorem}[Conditional Local Limit Theorem]\label{cll}
Suppose $\{X_n\}$ is a sequence of stationary Gaussian random variables with mean $0$ and covariance function $b(i-j)=E(X_iX_j)$ satisfying $b(t)=\dfrac{1}{2}[(t+1)^{2H}-2t^{2H}+(t-1)^{2H}]$, where $3/4<H<1$.  Then there exists a normalization sequence $\{d_n\}$, satisfying $d_n^2= n^{2H}L(n)$ as $n\to \infty, $ where $L(n)$ is  slowly varying and converging to a constant, such that for any interval $(a,b)\subset \mathbb R$ and any sequence $\{q_n\}$ and $\kappa \in \mathbb{R}$, such that $\dfrac{q_n}{d_n}\to \kappa$ as $n\to \infty$,  the conditional probability satisfies
\begin{equation*}
\lim_{n\to \infty}d_nP\bigg(S_n\in (q_n+a,q_n+b)|(X_{n+1},X_{n+2},\cdots) \bigg)=(b-a)g(\kappa), \text{ a.s.,} 
\end{equation*}
where  $g$ is the density function of the standard normal distribution.
 
 In  case that $q_n=0$, the convergence is uniform for   almost all $\omega$ and intervals $(a,b).$
\end{theorem}

Then the second main result of the paper follows:   the normalized occupation time of $S_n$ converges to   Mittag-Leffler distribution.

\begin{theorem}[Limiting distribution of the occupation time of $S_n$]\label{main}
Let $(X_n)_{n\in\mathbb Z}$ be as in Theorem \ref{cll}.  Denote the occupation time of $S_k$ in the  interval $(a,b)$ at time $n$ by $\ell_n([a,b])=\sum_{i=1}^n\mathbb 1_{(a,b)}(S_i)$. Then there exists a sequence of numbers $a_n=O(n^{1-H})$ such that
\begin{eqnarray}\label{covg}
&&\frac{1}{2\epsilon} \int_{-\epsilon}^{\epsilon}\left[\int_{\Omega} v\bigg(\dfrac{\ell_n([a-x,b-x])}{a_n}\bigg)  \Phi(\omega)dP(\omega)\right] \, dx\to E[v((b-a)Y_{\alpha})],
\end{eqnarray}
for any $\epsilon>0$, any bounded and continuous function $V:\mathbb R\to \mathbb R$,  any probability density function $\Phi \in L^2( P)$, where  $Y_{\alpha}$ is    a random variable having the Mittag-Leffler distribution with index $\alpha=1-H$.
\end{theorem}
\begin{remark} (1) Taking $\Phi=1$ one could try to evaluate the left hand side when $\epsilon\to 0$. This could show that the occupation times have a weak limit which is Mittag-Leffler. We do not know this, but the result shows that convergence in  the weak* sense in $L^\infty(dx)$.

(2) We do not know the precise connection of this result to the local time of fractional Brownian motion. In \cite{KY} it is remarked that the law of the local time of a fractional Brownian motion  is not a Mittag-Leffler distribution unless it is Brownian motion, although Kono's result in \cite{Kono}  suggested that it may be true. Theorem \ref{main} may give a hint to explain this phenomenon. Kasahara and Matsumoto have found that the limiting distribution of the occupation time of $B^H$  is similar but not equal to a  Mittag-Leffler distribution.  In the proof that the limiting distribution is not Mittag-Leffler, it is assumed that $d\geq 2$, $0<Hd<1$. However, their proof is still available when $H\neq\dfrac{1}{2}$ and  $d=1$.
\end{remark}

This paper is structured as follows. Section \ref{sectioncllt} is devoted to proving the conditional local limit theorem of $S_n$. In Section \ref{section3}, the occupation time of $S_n$ is represented as an ergodic partial sum by introducing a skew product transformation $\tilde{T}$ on $\Omega\times \mathbb{R}$. The skew product has an ergodic decomposition, and each component is pointwise dual ergodic.  It follows that the normalized occupation time of $S_n$ converges to    Mittag-Leffler distribution as described in Theorem \ref{main}.

\section{Conditional Local Limit Theorem}\label{sectioncllt}

\subsection{Proof of the conditional local limit theorem}

In this part, we state two claims   which are the key points  in the proof of Theorem \ref{cll} and provide the proof modulo these conditions.

\begin{proof}{ of Theorem \ref{cll}}.
For fixed  positive integers $k$ and $n$, the conditional probability $P\big(S_n\in(q_n+a,q_n+b)|(X_{n+1},X_{n+2},...X_{n+k})\big)$ is given by  a normal distribution.
Indeed,
let  the $(n+k)$-dimensional random variable $X=(X_1, \cdots,X_{n+k})^T$ be partitioned as $
\begin{bmatrix}
X_{[1]}\\
X_{[2]}
\end{bmatrix}
$
with sizes $n$ and $k$ respectively.
The covariance matrix of $X$ is denoted by
$$
\Sigma=
\begin{bmatrix}
\Sigma_{11}(n,n) &\Sigma_{12}(n,k)\\
\Sigma_{21}(k,n) & \Sigma_{22}(k,k)
\end{bmatrix}
,$$
where $\Sigma_{11}$ and $\Sigma_{22}$ are symmetric Toeplitz matrixes, 
$\Sigma_{11}(n,n)=[b(|i-j|)]_{0\le i,j<n}$,
$\Sigma_{22}(k,k)=[b(|i-j|)]_{0\le i,j<k},$
and  where
$\Sigma_{12}(n,k)=\Sigma_{21}^{T}(k,n)= [b(n-i+j)]_{0\le i<n; 0\le j<k}$.

Let $D$ be the $(k+1) \times(n+k)$ matrix, defined by
$D=
\begin{bmatrix}
    e(n)&0,\cdots,0\\
    0&I_k
\end{bmatrix}
$, where 
$e(n)=\underbrace{(1,1,...,1)}_{n}$ and let $I_k$ be the identity matrix of dimension $k$. Then  $DX\sim \mathcal{N}(0, D\Sigma D^T)$, i.e. 
$$\begin{bmatrix}
S_n\\

X_{[2]}
\end{bmatrix}
\sim \mathcal{N}\bigg(0, 
\begin{bmatrix}
e(n)\Sigma_{11}e(n)^T& e(n)\Sigma_{12}\\
\Sigma_{21}e(n)^T& \Sigma_{22}
\end{bmatrix}
\bigg).$$ By the  conditional normal formula (see for example \cite{MNC}, Section 5.5), when $\Sigma_{22}$ is of full rank, 
$$(S_n|X_{[2]})\sim \mathcal{N}\big(\mu(n,k),\sigma^2(n,k)\big),$$ 
where  
\begin{equation*}
\mu(n,k)=e(n)\Sigma_{12}(n,k) \Sigma_{22}^{-1}(k,k)X_{[2]}
\end{equation*}
and
\begin{equation*}
\sigma^2(n,k)=e(n)\Sigma_{11}(n,n)e(n)^T-e(n)\Sigma_{12}(n,k)\Sigma^{-1}_{22}(k,k)\Sigma_{21}(k,n)e(n)^T.
\end{equation*}
That is, $P(S_n\in A|X_{[2]})=\int_{A} f(y_1|X_{[2]})dy_1$, with $f(y_1|X_{[2]})=\dfrac{1}{\sqrt{2\pi}\sigma(n,k)}e^{-\frac{(y_1-\mu(n,k))^2}{2 \sigma^2(n,k)}}$.

Let $B=(B(1), B(2), \cdots, B(k))^T:=\Sigma_{21}(k,n)e(n)^T,$ so $B(s)=\sum_{i=s}^{n+s-1}b(i)$.
Then the mean and the variance become
$$\mu(n,k)=B^T\Sigma_{22}^{-1}(k,k)X_{[2]}$$ and 
 $$\sigma^2(n,k)=e(n)\Sigma_{11}e(n)^T-B^T\Sigma^{-1}_{22}B.$$ 
 
 It follows that
\begin{eqnarray*}
&&\sigma(n,k)P\bigg(S_n\in(q_n+a,q_n+b)|(X_{n+1},X_{n+2},...X_{n+k})\bigg)\\
&=&\dfrac{1}{\sqrt{2\pi}} \int_{a+q_n}^{b+q_n} \exp\bigg(-\frac{(x-B^T \Sigma_{22}^{-1}(k,k)X_{[2]})^2}{2\sigma^2(n,k)}\bigg)dx\\
&=&\sigma(n,k)\dfrac{1}{\sqrt{2\pi}} \int_{(a+q_n)/\sigma(n,k)}^{(b+q_n)/\sigma(n,k)} \exp\bigg(-\dfrac{1}{2}(y-\dfrac{1}{\sigma(n,k)}B^T \Sigma_{22}^{-1}(k,k)X_{[2]})^2\bigg)dy\\
&=&\sigma(n,k)\dfrac{1}{\sqrt{2\pi}} (b-a)/\sigma(n,k) \exp\bigg(-\dfrac{1}{2}(\frac{\xi+q_n}{\sigma(n,k)}-\dfrac{1}{\sigma(n,k)}B^T \Sigma_{22}^{-1}(k,k)X_{[2]})^2\bigg),
\end{eqnarray*}
where  the mean value theorem  is used in the last step  and $\xi\in [a,b]$. 

Now we make two claims which will be proved in  Sections \ref{var} and \ref{mean} below.

\begin{description}
\item[Claim 1:]\label{Claim1} For fixed $n$,  $d_n^2=\displaystyle{\lim_{k\to \infty}}\sigma^2(n,k)$ exists and   
$d_n^2=n^{2H}L(n)$ as $n\to \infty,$ where $L(n)$ is slow varying and converges to a constant.

\item[Claim 2:] \label{Claim2}
For fixed $n$, $
\displaystyle{\lim_{k\to \infty}}\dfrac{1}{\sigma(n,k)}B^T\Sigma_{22}^{-1}(k,k)X_{[2]}= 0$
almost surely.
\end{description}

 As a consequence,
$
\displaystyle{\lim_{k\to \infty}}\frac{\xi+q_n}{\sigma(n,k)}=\frac{\xi}{d_n}+\frac{q_n}{d_n}=:\kappa(n).
$
Since $\xi\in[a,b]$ and $\frac{q_n}{d_n}\to \kappa$ as $n\to \infty$,
$
\lim_{n\to \infty}\kappa(n)=\kappa.
$

Hence
 \begin{eqnarray*}
&&\lim_{k\to \infty}\sigma(n,k)P\bigg(S_n\in(q_n+a,q_n+b)|(X_{n+1},X_{n+2},...X_{n+k})
\bigg)\\&=&\dfrac{1}{\sqrt{2\pi}}(b-a)\exp(-\dfrac{\kappa(n)^2}{2})=g(\kappa(n))(b-a),
\end{eqnarray*}
where $g$ is the density function of the standard normal random variable.

On the other hand, by Doob's martingale convergence theorems, almost surely,
$$\lim_{k\to \infty}\sigma(n,k)P\bigg(S_n\in(q_n+a,q_n+b)|(X_{n+1},X_{n+2},...,X_{n+k})
\bigg)=d_nP\bigg(S_n\in(q_n+a,q_n+b)|(X_{n+1},X_{n+2},...)
\bigg).$$
Hence almost surely,
$$d_nP\bigg(S_n\in(q_n+a,q_n+b)|(X_{n+1},X_{n+2},...)
\bigg)=(b-a)g(\kappa(n))$$
It follows that
$$\lim_{n\to \infty}d_n P\bigg(S_n\in(q_n+a,q_n+b)|(X_{n+1},X_{n+2},...)
\bigg)=g(\kappa)(b-a).$$

When $q_n=0$, almost surely,
$$d_nP\bigg(S_n\in(a,b)|(X_{n+1},X_{n+2},...)
\bigg)=(b-a)g(0).$$
So
$$\lim_{n\to \infty}d_n P\bigg(S_n\in(a,b)|(X_{n+1},X_{n+2},...)
\bigg)=g(0)(b-a),$$
uniformly for almost all $\omega\in \Omega$ and $(a,b).$

In the following sections, we give the proof of the two claims. 
\end{proof}

\subsection{Estimating the variance}\label{var}
In this part, we prove Claim 1: $\displaystyle{\lim_{k\to \infty}}\sigma^2(n,k)=d^2_n=n^{2H}L(n)$. 
Since  by definition of the $b(i)$ we have that the first term of $\sigma^2(n,k)$, $e(n)\Sigma_{11}(n,n)e(n)^T=n^{2H}L(n)$, it is sufficient to prove that the second term of $\sigma^2(n,k)$ converges to $0$ as $k\to \infty$, i.e. $$\lim_{k\to \infty}B^T\Sigma^{-1}_{22}(k,k)B=0.$$  We shall write $\Sigma_{22}$ for $\Sigma_{22}(k,k)$ in this subsection to simplify notation.
First we give an estimate of the element $B(s)=\sum_{i=s}^{n+s-1}b(i)$ of the vector $B$.
\begin{lemma}\label{BS}  It holds that
 $$B(s)=\binom{2H}{2} n(s+n)^{2H-2}\bigg(1+O(\dfrac{1}{s})\bigg), \text{ as $s\to \infty$}$$
 and therefore  $\sum_{i=1}^kB^2(i)=O(n^2 (k+n)^{4H-3})$ as $k\to \infty.$
 \end{lemma}
 
 \begin{proof} By Taylor expansion, $(1+x)^a=\sum_{i=0}^{\infty} \binom {a}{i}x^{i}$ when $|x|<1$, where $\binom{a}{i}=\dfrac{a(a-1)\cdots(a-i+1)}{i!}$.   So by definition of $B(s)$ and $b(t)$,
 \begin{eqnarray}\label{eq:2.1}
 && 2 (sn)^{-2H}B(s)= (sn)^{-2H} \left[(s+n)^{2H}-(s+n-1)^{2H} -s^{2H}+(s-1)^{2H}\right] \\
 &=& \left(1-\frac{sn-s-n}{sn}\right)^{2H}-\left(1-\frac{sn-s-n+1}{sn}\right)^{2H} -\left(1-\frac{sn-s}{sn}\right)^{2H}+\left(1-\frac{sn-s+1}{sn}\right)^{2H} \notag\\
  &=& \sum_{i=2}^{\infty} \binom {2H} {i}(-1)^i f(s,n,i)(sn)^{-i}\notag
 \end{eqnarray}
 where $f(s,n,i)=(sn-s-n)^i-(sn-s-n+1)^i-(sn-s)^i+(sn-s+1)^i.$
 Using  the binomial formula, we can rewrite 
 $f(s,n,i)=\sum_{j=1}^i \binom{i}{j}\bigg( (sn-s)^{i-j}-(sn-s-n)^{i-j}\bigg).$
Since 
\begin{eqnarray*}
 &&(sn-s)^{i-j}-(sn-s-n)^{i-j}= \sum_{l=1}^{i-j}\binom{i-j}{l}(sn-s-n)^{i-j-l}n^l\\
 &=&n(i-j) (sn-s-n)^{i-j-1}\left(1+O\left(\dfrac{n}{sn-s-n}\right)\right), \text{as $\dfrac{n}{sn-s-n}\to 0,$}\\
\end{eqnarray*}
a straight forward calculation furthermore shows that
\begin{eqnarray*}
f(s,n,i)=i(i-1)n(sn-s-n)^{i-2}\left(1+O\left(\dfrac{n}{ns-s-n}\right)\right)\\
\end{eqnarray*}
 as $\dfrac{n}{sn-s-n}\to 0$ and $\dfrac{1}{sn-s-n}\to 0.$

Inserting into equation (\ref{eq:2.1}) we arrive at
 \begin{eqnarray*}
 && 2 (sn)^{-2H}B(s)= \sum_{i=2}^{\infty} \binom {2H} {i}(-1)^ii(i-1)n(sn-s-n)^{i-2}(sn)^{-i}(1+O(\dfrac{n}{ns-s-n}))\\
 &=&(2H)(2H-1) \dfrac{1}{s^2n} (\frac{1}{s}+\dfrac{1}{n})^{2H-2}\left(1+O\left(\dfrac{n}{ns-s-n}\right)\right).
  \end{eqnarray*}
 This shows that
$B(s)=\binom{2H}{2} n(s+n)^{2H-2}(1+O(\dfrac{1}{s}))$ as $s\to \infty.$ \QEDA
\end{proof}
 
  The main idea of estimating $B^T\Sigma_{22}^{-1}B$ is to write
\begin{eqnarray*}
B^T\Sigma_{22}^{-1}B=c(k)B^TA^{-1}B=c(k)B^T\sum_{l=0}^{\infty} (I-A)^lB=c(k)\sum_{l=0}^{\infty}B^T (I-A)^lB,
\end{eqnarray*}
where $c(k)$ is an appropriate constant chosen below satisfying   $||I-A||_2<1$ for $A=c(k)\Sigma_{22}$.

We consider  the minimal and maximal eigenvalues of $\Sigma_{22}$, denoted by $\lambda_{\min}(\Sigma_{22})$  and $\lambda_{\max}(\Sigma_{22})$, before determining $c(k)$, which are closely related to the norm of $\Sigma_{22}^{-1}.$

\begin{lemma}
 Suppose $H>\frac{1}{2}$, then $\lambda_{\min}(\Sigma_{22})\to c=\mbox{\rm essinf}\, f>0$ as $k\to \infty$. 
\end{lemma}
\begin{proof} 
One can define\cite{Hong} the power spectrum (see \cite{CFS}, chapter 14) with a singularity at $\lambda=0$ by 
\begin{equation*} 
f(\lambda):=\sum_{k=-\infty}^{\infty}b(k)e^{-i2\pi \lambda k}, -\frac{1}{2}<\lambda<\frac{1}{2},
\end{equation*}
(also known as the spectral density function or spectrum of the stationary process $\{X_n\}$)
where $b(k)=E(X_nX_{n+k})$ is the covariance function as before. $f(\lambda)$ has an inverse transformation:
$b(k)=\int_{-1/2}^{1/2} f(\lambda) e^{2\pi ik\lambda}d\lambda.$
For the fractional Gaussian noise $\{X_k\}$, $f(\lambda)$ has the form (cf. \cite{Hong}, Section 2.3):
\begin{equation*}
f(\lambda)=C|1-e^{i2\pi \lambda}|^2\sum_{m=-\infty}^{\infty}\dfrac{1}{|\lambda+m|^{1+2H}}=C(1-\cos(2\pi \lambda))\sum_{m=-\infty}^{\infty}\dfrac{1}{|\lambda+m|^{1+2H}}.
\end{equation*}

From \cite{Toeplitz} page 64/65, $\lim_{k\to \infty}\lambda_{\min}(\Sigma_{22})= \mbox{\rm essinf}\, f$ and $\lim_{k\to \infty}\lambda_{\max}(\Sigma_{22})=\mbox{\rm esssup}\, f$ (including the cases $\pm\infty$). Since $\mbox{\rm essinf}_{\lambda}\,f(\lambda)>0$, the lemma is proved.
 \QEDA
\end{proof}

\begin{lemma}
Let $c(k)=\frac{1}{mk^{2H-1} }$, where m is a  constant independent of $k$. If $m$ is large enough then $||I-A||_2=||I-c(k)\Sigma_{22}||_2<1$.
\end{lemma}
\begin{proof}
By Lemma \ref{BS}, on the one hand,
$
||\Sigma_{22}||_2\leq \sqrt{||\Sigma_{22}||_1||\Sigma_{22}||_{\infty}}=O(k^{2H-1}),
$  and on the other hand, $\dfrac{B^T\Sigma_{22}B}{B^TB}\leq \lambda_{\max}(\Sigma_{22})=||\Sigma_{22}||_2$ and  $\dfrac{B^T\Sigma_{22}B}{B^TB}=O(k^{2H-1})$. Hence
$||\Sigma_{22}||_2=O(k^{2H-1}).$

The eigenvalues  of $I-A:=I-c(k)\Sigma_{22}$ are $\{1-c(k)\lambda_{i}(\Sigma_{22})\}$, $\lambda_i(\Sigma_{22})$ denoting the eigenvalues of $\Sigma_{22}$. Therefore, choosing $m$ large enough,  $|1-c(k)\lambda_{\max}(\Sigma_{22})|$ and $|1-c(k)\lambda_{\min}(\Sigma_{22})|$ can be both made less than 1, idependently of $k$. Hence $||I-c(k)\Sigma_{22}||_2<1$. \QEDA
 \end{proof}
 
 With the preparation above, we can return to the estimate of $B^T\Sigma_{22}^{-1}B=c(k)\sum_{l=0}^{\infty}B^T (I-c(k)\Sigma_{22})^lB$.

\begin{lemma}\label{cov}
If  $\frac{3}{4}<H<1$, then
\begin{equation*}
\lim_{k\to \infty}B^T\Sigma_{22}^{-1}B=0.
\end{equation*}
\end{lemma}
It follows from the lemma that $d_n^2:=\lim_{k\to\infty}\sigma^2(n,k)=O(n^{2H})$.
\begin{proof}
We first derive an recursive equation, which will be used frequently.

For any k-dimensional column vector $V=(V(1), V(2), \cdots, V(k))^T$,  define $V^{(l)}:=(I-c(k)\Sigma_{22})^lV=(I-c(k)\Sigma_{22})V^{(l-1)}$,  $V^{(0)}=V$. Recall that $B^T=(B(1), B(2), \cdots, B(k))$, then
\begin{eqnarray*}
B^TV^{(l)}&=&B^T(I-c(k)\Sigma_{22})V^{(l-1)}\\
&=&B^TV^{(l-1)}-c(k)B^T\Sigma_{22}V^{(l-1)}\\
&=& \sum_{s=1}^k V^{(l-1)}(s) B(s)\bigg(1- c(k) \sum_{i=1}^k\frac{B(i)}{B(s)}b(i-s)\bigg).
\end{eqnarray*}
Hence we get a recursive equation for any vector $V$:
\begin{eqnarray}\label{induction}
\sum_{s=1}^k B(s) V^{(l)}(s)&=& \sum_{s=1}^kB(s) V^{(l-1)}(s) \bigg(1-q(s)\bigg), \text{$l\geq 1$}
\end{eqnarray}
where
\begin{equation*}\label{q}
q(s)= c(k)\sum_{i=1}^k\frac {B(i)}{B(s)}b(i-s)=c(k)\left(1+\sum_{1\leq i\leq k, i\neq s}\frac {B(i)}{B(s)}b(i-s)\right)
\end{equation*}

By Lemma \ref{BS} and $b(t)=\binom{2H}{2} |t|^{2H-2}(1+\frac{\binom{2H}{4}}{\binom{2H}{2}}\frac{1}{t^2}+\cdots)$ for $t\geq 1,$ 
\begin{eqnarray*}
q(s)&\geq& c(k) \bigg(\int_{2}^{s-2} \frac{B(i)}{B(s)}b(s-i)di +\int_{s+2}^{k-1} \frac{B(i)}{B(s)}b(s-i)di\bigg)\\
&\geq& c(k) \bigg(C\int_{2}^{s-2} \frac{(i+n)^{2H-2}}{(s+n)^{2H-2}}(s-i)^{2H-2}di +\int_{s+2}^{k-1} \frac{(i+n)^{2H-2}}{(s+n)^{2H-2}}(i-s)^{2H-2}di\bigg)\\
&=&\frac{C}{m} (\frac{n+s}{k})^{2-2H} \bigg(\int_{\frac{2}{k}}^{\frac{s-2}{k}} (x+\frac{n}{k})^{2H-2}(\frac{s}{k}-x)^{2H-2}dx+\int_{\frac{s+2}{k}}^{\frac{k-1}{k}} (x+\frac{n}{k})^{2H-2}(x-\frac{s}{k})^{2H-2}dx\bigg) \\
&=&\frac{1}{m} (\frac{n+s}{k})^{2-2H}  I(k,s,n).
\end{eqnarray*}
The integral $I(k,s,n)$ is bounded below by some constant $K$ independent of $k, s$ and $n$. 

For all $s$ satisfying 
\begin{equation*}\label{ineq}
 (s+n)^{2-2H}\geq n^{2-2H}k^{\gamma(2-2H)}
\end{equation*}
where $\gamma\in[0,1]$ is determined below,
that is
$$ s\geq s^*:=nk^{\gamma}-n ,$$
we get
$$ q(s)\geq q:=\frac{n^{2-2H}}{k^{(2-2H)(1-\gamma)}} \frac{K}{m}, \text{ when $s\geq s^{*}$.}$$ 

In the recursive equation (\ref{induction}), put  $V(s)=B(s), s=1,2,\cdots, k$, so $V^{(l)}=(I-c(k)\Sigma_{22})^lB=(I-c(k)\Sigma_{22})V^{(l-1)}$. Then one has
\begin{eqnarray*}
 \sum_{s=1}^k B(s) V^{(l)}(s) &=&  \sum_{s=1}^k \big(1-q(s)\big)B(s)V^{(l-1)}(s)\\
&\leq & \sum_{s=1}^k (1- q) B(s) V^{(l-1)}(s)+  \sum_{s< s^*} (q-q(s))B(s)V^{(l-1)}(s)\\
&\leq&  (1-q)\sum_{s=1}^k B(s)V^{(l-1)}(s)  + (q-\min_{s<s^*}q(s))\sum_{s<s^*}B(s)V^{(l-1)}(s).
\end{eqnarray*}

The idea is to incorporate the second term (when $s<s^*$)   into the first one.
For any $\epsilon>0$, define 
\begin{eqnarray*}\label{second term}
L^*=\inf\{l: \sum_{s=1}^{s^*-1}B(s)V^{(l)}(s)\geq \epsilon \sum_{s=1}^k B(s)V^{(l)}(s)\}.
\end{eqnarray*}

If $L^*=\infty$, then $\sum_{s=1}^{s^*-1}B(s)V^{(l)}(s) < \epsilon \sum_{s=1}^k B(s)V^{(l)}(s)$ for all $l$. If $\epsilon$ is small enough, it follows that for all $l\geq 1,$
\begin{eqnarray*}
 \sum_{s=1}^k B(s) V^{(l)}(s) &\leq&  (1-q) \sum_{s=1}^k B(s)V^{(l-1)}(s)  \\
&+&  \epsilon(q-\min_sq(s)) \sum_{s=1}^k B(s)V^{(l-1)}(s)\\
&=&(1-q^*) \sum_{s=1}^k B(s)V^{(l-1)}(s)\\
&\leq&(1-q^*)^l\sum_{s=1}^k B(s)V^{(0)}(s)
\end{eqnarray*}
where $q^*=q-\epsilon(q-\min_sq(s))$. By Lemma \ref{BS}, for some constants $C_1, C_2>0$
\begin{eqnarray*}
c(k) \sum_{l=1}^{\infty}\sum_{s=1}^k B(s) V^{(l)}(s) &\leq&c(k)\sum_{l=1}^{\infty}(1-q^*)^l\sum_{s=1}^{k} B(s)^2\\
&\leq&  c(k)  \frac{C_1}{q^*}n^2(k+n)^{4H-3}\\
&\leq& c(k) \frac{C_1}{(1-\epsilon)q}n^2(k+n)^{4H-3}\\
&=& C_2 n^{2H} k^{(2-2H)(-\gamma)}.
\end{eqnarray*}
When $l=0$,  for some constant $C_3>0$
$$c(k) \sum_{s=1}^k B(s) V^{(0)}(s)=c(k)\sum_{s=1}^k B^2(s)\leq c(k) mC_3(n^2(k+n)^{4H-3})\leq C_3(k^{-2(1-H)}n^{2}).$$
Hence,
$$B^T\Sigma_{22}^{-1}B=c(k)\sum_{l=0}^{\infty}B^TV^{(l)}= O(k^{-2\gamma(1-H)}), \text{ as $k\to \infty.$}$$
Therefore,  if $L^*=\infty$, then $\lim_{k\to \infty}B^T\Sigma_{22}^{-1}B=0.$

If $L^*<\infty$, then for $l<L^*$, 
\begin{equation*}
\sum_{s=1}^{s^*-1}B(s)V^{(l)}(s)< \epsilon \sum_{s=1}^k B(s)V^{(l)}(s),
\end{equation*}
and 
\begin{equation}\label{ineq}
\sum_{s=1}^{s^*-1}B(s)V^{(L^*)}(s)\geq \epsilon \sum_{s=1}^k B(s)V^{(L^*)}(s).
\end{equation}
In this case, we split $\sum_{l=1}^{\infty}B^T V^{(l)}$ into two parts:
\begin{eqnarray*}
 \sum_{l=1}^{\infty}\sum_{s=1}^k B(s) V^{(l)}(s) &=& \sum_{l=1}^{L^*-1}\sum_{s=1}^k B(s) V^{(l)}(s)+\sum_{l=L^*}^{\infty}\sum_{s=1}^k B(s) V^{(l)}(s).
 \end{eqnarray*}
The first term can be handled with  in the same way as the case when $L^*=\infty$:
 \begin{eqnarray}\label{term1}
c(k) \sum_{l=1}^{L^*-1}\sum_{s=1}^k B(s) V^{(l)}(s)&\leq&c(k)  \sum_{l=1}^{L^*-1} (1-q^*)^l \sum_{s=1}^k B^2(s)\notag\\
&\leq& c(k)\frac{1}{q^*}\sum_{s=1}^k B^2(s)\\
&\leq&C_2 n^{2H} k^{-2\gamma(1-H)}.\notag
 \end{eqnarray}
 For the other term, in the iteration equation (\ref{induction}) take $V$ to be $V_{new}:=(I-c(k)\Sigma_{22})^{L^*}B=V^{(L^*)}$, then by changing variable $d=l-L^*$, one has
 \begin{eqnarray*}
 \sum_{l=L^*}^{\infty} B^T V^{(l)}&=&\sum_{l=L^*}^{\infty} B^T(I-c(k)\Sigma_{22})^{l-L^*} V_{new}
\\
&=&\sum_{l=L^*}^{\infty} B^TV^{(l-L^*)}_{new}
\\
&=&\sum_{d=0}^{\infty} \sum_{s=1}^{k}B(s) V_{new}^{(d)}(s)
\\
&=&\sum_{s=1}^{k}B(s) V_{new}^{(0)}(s)+\sum_{d=1}^{\infty} \sum_{s=1}^{k} (1-q(s)) B(s) V_{new}^{(d-1)}(s)\\
&\leq& \sum_{d=0}^{\infty} (1-\min_{s}q(s))^d\sum_{s=1}^{k}B(s)V^{(L^*)}(s) \\
&\leq&  \sum_{d=0}^{\infty} (1-\min_{s}q(s))^d \cdot\frac{1}{\epsilon}\sum_{s=1}^{s^*-1}B(s)V^{(L^*)}(s) \text{ by (\ref{ineq})} \\ 
&\leq& \frac{1}{\min_{s}q(s)\epsilon} \sum_{s=1}^{s^*-1}B(s)^2\\
&\leq& \frac{C_4}{\min_{s}q(s)} n^{4H-1}k^{\gamma(4H-3)},
\end{eqnarray*}
where $C_4>0$ and Lemma \ref{BS} is used.

Since $q(s)\geq \frac{1}{m} (\frac{n+s}{k})^{2-2H}K,$ $\min_sq(s)\geq \frac Km \frac{n^{2H-2}}{k^{2-2H}}$, one has
\begin{eqnarray}\label{term2}
c(k)\sum_{l=L^*}^{\infty} B^T V^{(l)}\leq c(k) \frac{C_4}{\min_{s}q(s)\epsilon} n^{4H-1}k^{\gamma(4H-3)}=O(k^{-(1-\gamma)(4H-3)}).
\end{eqnarray}

Combining ($\ref{term1}$) and ($\ref{term2}$), one has $B^T\Sigma_{22}^{-1}B\leq C(k^{-(1-\gamma)(4H-3)}+ k^{-2\gamma(1-H)})$.
Since $H\in(\frac{3}{4},1)$, $\lim_{k\to \infty}B^T\Sigma_{22}^{-1}B=0$ follows.\QEDA
\end{proof} 

\subsection{Estimate of the mean}\label{mean}

We continue writing $\Sigma_{22}$ for $\Sigma_{22}(k,k)$.
\begin{lemma}
Suppose $3/4<H<1$, then almost surely, 
\begin{equation*}
\lim_{k\to \infty}B^T\Sigma^{-1}_{22}X_{[2]}=0,
\end{equation*}
where $X_{[2]}=(X_{n+1}, X_{n+2}, \cdots, X_{n+k})^T$ and  $B=(B(1),...,B(k))^T$ depend on $k$ as before and $n$ is fixed.
\end{lemma}
\begin{proof}
The random variable $B^T\Sigma^{-1}_{22}X_{[2]}$ has a normal distribution. Its mean  and  variance are $0$ and  $(B^T\Sigma_{22}^{-1})\Sigma_{22}(\Sigma_{22}^{-1}B)=B^T\Sigma_{22}^{-1}B$, respectively.
Hence  for any $\epsilon>0$, and surpressing the running index $k$,
\begin{eqnarray*}
\sum_{k=1}^{\infty} P(|B^T\Sigma^{-1}_{22}X_{[2]}|>\epsilon)
&=&\frac{2}{\sqrt{2\pi}}\sum_{k=1}^{\infty} \int_{\epsilon/(B^T\Sigma_{22}^{-1}B)^{\frac{1}{2}}}^{\infty}e^{-\frac{x^2}{2}}dx\\
&\leq&\frac{2}{\sqrt{2\pi}}\sum_{k=1}^{\infty} \exp\left\{-\frac{(\epsilon/(B^T\Sigma_{22}^{-1}B)^{\frac{1}{2}})^2}{2}\right\}\\
&=&\frac{2}{\sqrt{2\pi}}\sum_{k=1}^{\infty} \exp{(-\frac{\epsilon^2}{2}(B^T\Sigma_{22}^{-1}B)^{-1})}.
\end{eqnarray*}
Recall that by the proof of Lemma  \ref{cov} $B^T\Sigma_{22}^{-1}B =O(k^{-\overline \gamma})$ for some $\overline \gamma>0$, so that by
the Borel-Cantelli Lemma it follows that $B^T\Sigma^{-1}_{22}X_{[2]}\to 0$ as $k\to \infty$ almost surely.\QEDA
\end{proof}

\section{Limit Theorem of the Occupation Times of $\{S_n\}$}\label{section3}

Recall that the occupation time of  $S_n=\sum_{i=1}^{n}X_i$ is defined as $\lambda(n,A)=\sum_{i=1}^n \mathbb 1_{A}(S_{i})$.
The main result  is that for some $a_n$, the normalized occupation time $\frac{1}{a_n}\lambda(n, A)$  converges to   Mittag-Leffler distribution in the sense of Theorem \ref{main}. In this section, we give the proof of Theorem \ref{main} via Theorem \ref{cll}.

\subsection{Representation of the occupation time }\label{represent}
 
 In this section, we consider the stationary Gaussian random variables $\{X_n\}$ as above to be represented as the coordinate process of a shift dynamical systems.

Without loss of generality,  suppose the random variables $\{X_n\}$ are defined on  the probability space $(\Omega,\mathcal B_0,P)$ with $\Omega=\mathbb{R}^{\mathbb{N}}$ and $\mathcal B_0$ is the $\sigma-$algebra generated by the cylinder sets of $\mathbb{R}^{\mathbb{N}}$. Let $T$ denote the shift operator: $T: \Omega \to \Omega$, $(T\omega)_{i}=\omega_{i+1}$ where $\omega=(\omega_1,\omega_2...)\in \mathbb{R}^{\mathbb{N}}$. Define $\phi: \Omega\mapsto \mathbb{R}$ as $\phi(\omega):=\omega_1$,  $\int_{\Omega}|\phi| dP<\infty$ and $\int_{\Omega}\phi ~dP=0$.  The probability  $P$ is the distribution of the  stochastic process $\{X_n\}$, which then is represented as $X_{n}(\omega):=\phi \circ T^{(n-1)}(\omega)=\omega_{n}$, $n\in \mathbb N$, and  has the joint Gaussian probability distribution with zero mean: for any family of Borel sets $C_1,C_2...C_r\subset \mathbb{R}$, 
$P(\{\omega\in \Omega: X_{n_1}(\omega) \in C_1,X_{n_2}(\omega) \in C_2,...,X_{n_r}(\omega)\in C_r\})=\int_{C_1 \times C_2 ...\times C_r}p(t_1,t_2,...,t_r)dt_1dt_2...dt_r$.
Here $p$ is the normal probability density function: $p(t_1,t_2,...,t_r)=C\exp(-\frac{1}{2}(Dt,t))$, where $D$ is the matrix inverse to the covariance matrix $B=(b({n_i- n_j}))$ and $b(n_i-n_j)=E[X_{n_i}X_{n_j}]$.

We represent the occupation time of $\{S_n\}$ by introducing the skew product:
Let $(X,\mathcal{B},\mu)=(\Omega \times \mathbb{R},\mathcal B_0 \times \sigma(\mathbb{R}),P \times m_{\mathbb{R}})$, where $\sigma(\mathbb{R})$ is the Borel $\sigma$-algebra and $m_{\mathbb{R}}$ denotes the Lebegue measure on $\mathbb{R}$. Define $\tilde{T}:\Omega \times \mathbb{R}\to \Omega \times \mathbb{R} $ by $\tilde{T}(\omega, r):=(T(\omega),r+\phi(\omega))$. By induction, $\tilde{T}^n(\omega, r)=(T^n\omega, r+S_n(\omega))$, where $S_n$ is the partial sum of $\{X_n\}$. 

Define $S^{\tilde{T}}_n(f):=\sum_{k=0}^{n-1}f\circ {\tilde{T}}^k$ for any $f:\Omega\to \mathbb{R}$. Specifically, $S^T_n(\phi)= S_n$. Let $A=\Omega \times \mathbb 1_{D}$, then the occupation time of $\{S_n\}$ for set $D\subset\mathbb{R}$ has the following representation: \\
\begin{equation*}
\lambda(n,D)=\sum_{i=1}^n \mathbb 1_{\{S_i\in D\}}=\sum_{i=1}^n \mathbb 1_{\{ A\}}(\tilde{T}^i(\omega, 0) )=S_n^{\tilde{T}}(\mathbb 1_{\{ A\}})(\omega, 0).
\end{equation*}

\subsection{Proof of the main theorem}\label{4}

\begin{theorem}[Conservative and Ergodic Decomposition]\label{ced}
For the dynamic system $(X, \mathcal{B}, \mu, \tilde{T})$ defined above, one has
\begin{enumerate}
\item $\tilde{T}$ is a conservative  measure preserving transformation of $(X,\mathcal{B},\mu)$.
\item There exists a probability space $(Y, \mathcal{C}, \lambda)$ and a collection of measures 
$\{\mu_{y}:y \in Y\}$ on $(X,\mathcal{B})$ such that 
\begin{enumerate}
\item for $\lambda-$almost all $y \in Y$, $\tilde{T}$ is a conservative, ergodic measure-preserving transformation of $(X, \mathcal{B}, \mu_y)$;
\item for $A\in \mathcal{B}$, the map $y\to \mu_{y}(A)$ is measurable and 
$$\mu(A)=\int_{Y} \mu_{y}(A)d\lambda(y).$$
\end{enumerate}
\item And $\lambda$-almost surely, $(X, \mathcal{B}, \mu_y, \tilde{T})$  is pointwise dual ergodic.
\end{enumerate}
\end{theorem}

\begin{proof}
\begin{enumerate}
\item 
By Corollary 8.1.5 in \cite{Jon}, in order to prove that $\tilde{T}$ is conservative, it is sufficient to show  that 
\begin{description}
\item[(1)]  $\phi:\Omega\to \mathbb{R}$ is integrable, and $\int_{\Omega}\phi ~dP=0$.
\item[(2)] $T$ is ergodic and probability  preserving on $(\Omega, \mathcal B_0,P)$.
\end{description}
\textbf{(1)} holds by our assumption on $\{X_n\}$.
For \textbf{(2)}, by \cite{CFS} (page 369), $\displaystyle{\lim_{n\to \infty}}b(n)=0$ is a necessary and sufficient   condition that $T$ is mixing: $|P(A\cap T^{-n}B)- P(A)P(B)| \to 0 $ as $n\to \infty$ for any $A,B\in \mathcal B_0$. It implies that $T$ is ergodic. $T$ is also a probability preserving transformation, since $\{X_n\}$ is a stationary process.   Hence $\tilde{T}$ is conservative.

$\tilde{T}$ is measure preserving since $T$ is measure preserving.
\item
The proof of the ergodicity decomposition is an adaption of the corresponding argument of Section 2.2.9 of \cite{Jon} (page 63).
\item
Next we prove that $(X, \mathcal{B}, \mu_y, \tilde{T})$  is pointwise dual ergodic.

We claim that
\begin{equation*}\label{c.e.}
\sum_{n=0}^{\infty} P_{{\tilde{T}}^n} \mathbb 1_{\Omega} \otimes \mathbb 1_{(a,b)}=\infty\quad \mu-a.s.,
\end{equation*}
where $P_{\tilde{T}^n}: L^1(\mu)\to L^1(\mu)$ is the Frobenius-Perron operator of $\tilde{T}^n$: 
\begin{equation*}
\int_{X}(P_{\tilde{T}^n}f)\cdot g~d\mu=\int_{X} f\cdot (g\circ \tilde{T}^n)~d\mu
\end{equation*}
for any $g\in L^{\infty}(\mu).$
Likewise we denote the Frobenius-Perron operator for $T^n$ by $P_{T^n}$.
Now,
\begin{equation*}
P_{\tilde{T}^n}  (\mathbb 1_{\Omega} \otimes \mathbb 1_{(a,b)})(\omega, z)=P_{T^n}\Bigg(\mathbb 1_{(z-b,z-a)}\bigg(S_n(\cdot)\bigg)\Bigg)(\omega),\ \mu-a.s. \  (\omega,z)\in X,
\end{equation*}
and
\begin{equation*}
P_{T^n}\Bigg(\mathbb 1_{(z-b,z-a)}\bigg(S_n(\cdot)\bigg)\Bigg)(\omega)=P(S_n\in(z-b,z-a)|T^n(\cdot)=\omega).
\end{equation*}
By Theorem  \ref{cll},
\begin{equation*}
\sum_{n=0}^N P_{\tilde{T}^n}\bigg( \mathbb 1_{\Omega}\otimes \mathbb 1_{(a,b)} \bigg )(\omega, z)=\sum_{n=0}^NP(S_n\in(z-b,z-a)|T^n(\cdot)=\omega) \sim \sum_{n=0}^N (b-a)\frac{g(0)}{d_n}\;\mu-a.s.\ (\omega,z)\in X.
\end{equation*}  
Let $a_N=\sum_{n=0}^N\dfrac{g(0)}{d_n}$, then $a_N \to \infty$ as $N\to \infty$, and  
$$\sum_{n=0}^N P_{\tilde{T}^n}\bigg( \mathbb 1_{\Omega}\otimes \mathbb 1_{(a,b)} \bigg )\sim a_N(b-a), \;\mu-a.s..$$
It follows that $\lambda$-a.s. $y$,
\begin{equation*}
\sum_{n=0}^N P_{\tilde{T}^n}\bigg( \mathbb 1_{\Omega}\otimes \mathbb 1_{(a,b)} \bigg )\sim a_N(b-a)\;\mu_{y}-a.s..
\end{equation*}
\end{enumerate}
Since for $\lambda$-a.e. $y$, $\tilde{T}$ is conservative on $(X, \mathcal{B}, \mu_y)$, by  Hurewicz's ergodic theorem (see for example, \cite{Jon}), one has for all $f\in L^1(\mu_{y})$, $\mu_{y}$-almost surely,
\begin{eqnarray*}
\dfrac{1}{a_N}\sum_{n=0}^N P_{\tilde{T}^n}f&\sim& \dfrac{(b-a)\sum_{n=0}^N P_{\tilde{T}^n}f}{\sum_{n=0}^N P_{\tilde{T}^n}\bigg( \mathbb 1_{\Omega} \otimes \mathbb 1_{(a,b)} \bigg)}\\
&\to& \dfrac{(b-a)\int_{\Omega\times \mathbb{R}}fd\mu_y}{\int_{\Omega\times \mathbb{R}} \mathbb 1_{\Omega}\otimes \mathbb 1_{(a,b)} d\mu_y}\\
&=& \frac{(b-a)}{\mu_{y}(\Omega\otimes (a,b))}\int_{\Omega\times \mathbb{R}}fd\mu_y.
\end{eqnarray*}

Since we may change the interval $(a,b)$ to be some other interval $(c,d)$, one finds that $\frac{(b-a)}{\mu_{y}(\Omega\otimes (a,b))}$ does not depend on $(a,b)$, hence $\frac{(b-a)}{\mu_{y}(\Omega\otimes (a,b))}$ is a constant, denoted by $C(y)$.

Thus, $(X, \mathcal{B}, \mu_{y}, \tilde{T})$  is pointwise dual ergodic with return sequence $a_nC(y)$.
\QEDA
\end{proof}

We end  the paper with the proof of Theorem \ref{main}.
\begin{proof}
Since $\tilde{T}$ is pointwise dual ergodic with respect to measure $\mu_{y}$, suppose $a_n$ is regularly varying with index $\alpha=1-H$ and has the same order as $\sum_{i=0}^n \frac{g(0)}{d_i}$, then by Corollary 3.7.3 in \cite{Jon}, $\dfrac{S_n^{\tilde{T}}}{C(y)a_n}$ 
converges strongly in distribution, i.e.,
\begin{equation*} 
\int_{X} v\bigg(\dfrac{S^{\tilde{T}}_n(f)(\omega,x)}{C(y)a_n}\bigg)h_y(\omega,x)d\mu_{y}(\omega,x)\to E[v(\mu_{y}(f)Y_{\alpha})],
\end{equation*}
or equivalently,
\begin{equation}\label{strongconv} 
\int_{X} v\bigg(\dfrac{S^{\tilde{T}}_n(f)(\omega,x)}{a_n}\bigg)h_y(\omega,x)d\mu_{y}(\omega,x)\to E[v(C(y)\mu_{y}(f)Y_{\alpha})],
\end{equation}
for any bounded and continuous function $v:\mathbb R\to\mathbb R$ and for any $h_y\in L^1(\mu_{y})$ and $\int_{X}h_yd\mu_{y}=1,$  where $S_n^{\tilde{T}}(f)=\sum_{i=1}^{n}{f\circ {\tilde{T}}^{i-1}}$, and $Y_{\alpha}$ has the normalized Mittag-Leffler distribution of order $\alpha=1-H$.

Let $f=\mathbb 1_{\Omega}\times \mathbb 1_{(a,b)}$, then $S^{\tilde{T}}_n(f)(\omega,x)=\sum_{i=1}^n\mathbb 1_{(a,b)}(x+S_i(\omega))$, which is the occupation time  of $S_n$ at time $n$ on interval $(a-x,b-x)$. Since $C(y)=\frac{b-a}{\mu_y(\mathbb 1_{\Omega}\otimes \mathbb 1_{(a,b)})}$, $C(y)\mu_{y}(\mathbb 1_{\Omega}\otimes \mathbb 1_{(a,b)})=\int_{\mathbb{R}} \mathbb 1_{(a,b)}dm_{\mathbb R}$, then the right hand side of (\ref{strongconv}) is simplified to be $E\bigg[v\bigg((b-a)Y_{\alpha}\bigg)\bigg]$.

 If $\Phi_1(\omega,x)$ is  any probability density function  on $(X,\mathcal{B}, \mu)$,
 for each $y$, define 
 \begin{eqnarray*}
 \phi_y(\omega,x)=
 \begin{cases}
 \frac{1}{\int_X \Phi_1(\omega,x)d\mu_y(\omega,x)} \Phi_1(\omega,x), &\int_X \Phi_1(\omega,x)d\mu_y(\omega,x)\neq 0;\\
 0, &\int_X \Phi_1(\omega,x)d\mu_y(\omega,x)= 0.
 \end{cases}
 \end{eqnarray*}
 $ \phi_y(\omega,x)$ is a density function on $(X,\mathcal{B}, \mu_y)$ for $y\in U$ where $U=\{y\in Y: \int_X \Phi_1(\omega,x)d\mu_y(\omega,x)\neq 0\}$.  
 By (\ref{strongconv}) and Theorem \ref{ced},  one has 
\begin{eqnarray*}
&&\int_{X} v\bigg(\dfrac{\sum_{i=1}^n\mathbb 1_{(a,b)}(x+S_i(\omega))}{a_n}\bigg)\Phi_1(\omega,x)d\mu(\omega,x)\\
&=& \int_U \int_X v\bigg(\dfrac{\sum_{i=1}^n\mathbb 1_{(a,b)}(x+S_i(\omega))}{a_n}\bigg)\Phi_1(\omega,x)d\mu_y(\omega,x) d\lambda(y)\\
&=& \int_U (\int_X \Phi_1(\omega,x)d\mu_y(\omega,x))\int_X v\bigg(\dfrac{\sum_{i=1}^n\mathbb 1_{(a,b)}(x+S_i(\omega))}{a_n}\bigg)\phi_y(\omega,x)d\mu_y(\omega,x) d\lambda(y)\\
&=& \int_Y (\int_X \Phi_1(\omega,x)d\mu_y(\omega,x))\int_X v\bigg(\dfrac{\sum_{i=1}^n\mathbb 1_{(a,b)}(x+S_i(\omega))}{a_n}\bigg)\phi_y(\omega,x)d\mu_y(\omega,x) d\lambda(y)\\
&\to &\int_Y  \mu_y(\Phi_1)  E[v((b-a)Y_{\alpha})]d\lambda(y) \text{ by  the dominated convergence theorem}\\
&=&E[v((b-a)Y_{\alpha})].
\end{eqnarray*}

Let $\Phi_1(x,\omega)=\frac{1}{2\epsilon}\mathbb 1_{\{-\epsilon,\epsilon\}}(x)\otimes \Phi(\omega)$ where $\epsilon>0$ and   $\Phi(\omega)$ is a probability density function on $(\Omega, \mathcal B_0, P)$.  Then one obtains  as $n\to \infty,$
\begin{eqnarray*}
&&\frac{1}{2\epsilon} \int_{-\epsilon}^{\epsilon}\left[\int_{\Omega} v\bigg(\dfrac{\sum_{i=1}^n\mathbb 1_{(a-x,b-x)}(S_i(\omega))}{a_n}\bigg)\Phi(\omega)dP(\omega)\right] ~dx\to E[v((b-a)Y_{\alpha})].
\end{eqnarray*}
\QEDA
\end{proof}

\end{document}